\documentclass[11pt]{article}
\usepackage{tikz}
\tikzstyle{vertex}=[circle, draw, inner sep=0pt, minimum size=5pt]

\usetikzlibrary{decorations.markings}

\usepackage[english]{babel}
\usepackage{times}
\usepackage{amsthm}
\usepackage{amssymb, verbatim}
\usepackage{amsmath}
\usepackage{eucal}
\usepackage{mathrsfs}

\usepackage{fdsymbol}

\newtheorem{theorem}{Theorem}
\newtheorem{corollary}{Corollary}

\newtheorem{proposition}[theorem]{Proposition}

\theoremstyle{definition}

\newtheorem{remark}[theorem]{Remark}
\newtheorem{da-scrivere}[theorem]{da-scrivere}

\def\H{{\cal H}}

\def\U{{\cal U}}

\begin{document}
\title{On a spectral property of one-dimensional representations of compact quantum groups}
\author{Stefano Rossi\footnote{E-mail: rossis@mat.uniroma2.it}
\\
Dipartimento di Matematica,\\ Universit\`a di Roma Tor
Vergata\\ Via della Ricerca Scientifica 1, I--00133 Roma, Italy.}
\date{}
\maketitle

\begin{abstract}
In the $C^*$-algebraic setting the spectrum of any group-like element of a compact quantum group is shown to be a closed subgroup of the one-dimensional torus. A number of consequences of this fact are then illustrated, along with a loose connection with the so-called
Kadison-Kaplansky conjecture.  
\end{abstract}
$$$$
For any unitary representation $u$ of a compact group $G$ acting on a finite-dimensional complex Hilbert space $\H$ the set $u(G)$ is by definition a compact subgroup of the unitary group $\U(\H)$. This applies in particular to one-dimensional representations, for which $u(G)$ may also be identified with 
the spectrum of $u$ understood as an element of  $C(G)$, the $C^*$-algebra of continuous functions on $G$. But then one can take a step further wondering if the property singled out above still holds true for compact {\it quantum} groups as well. It turns out that it does as long as
$C^*$-algebraic compact quantum groups are dealt with. The proof is actually rather quick, although not quite elementary, inasmuch as it uses a non-trivial result concerning the spectrum of the tensor product of two operators on a Hilbert space.  Even so, none of the references we know settles the question explicitly, and this is why 
this brief note has been written.\\

Henceforward by compact quantum group we shall always mean a compact quantum group in the sense of Woronowicz, as it was first defined in \cite{WoroCMP}. This is given by a pair $G=(Q_G, \Delta_G)$ where $Q_G$ is a unital $C^*$-algebra and $\Delta_G: Q_G\rightarrow Q_G\otimes Q_G$ is a coassociative  unital $^*$-homomorphism. A $n$-dimensional unitary representation of $G$ is then a unitary $u=(u_{i,j})$ in $M_n(Q_G)$ such that $\Delta_G(u_{i,j})=\sum_{k=1}^n u_{i,k}\otimes u_{k,j}$ for every $i,j=1, 2,\ldots, n$. Thus a one-dimensional representation is nothing but a unitary $u\in Q_G$ such that
$\Delta_G(u)=u\otimes u$. One-dimensional representations are often referred to as group-like elements, especially when they show up in purely algebraic contexts. Here follows the main result of the present note.
\begin{theorem}[]
If $G=(Q_G,\Delta_G)$ is a compact quantum group and $u\in Q_G$ is a unitary one-dimensional representation of it, then $\sigma(u)\subset\mathbb{T}$ is a closed subgroup.
\end{theorem}
\begin{proof}
Since $\sigma(u)$ is a closed subset of $\mathbb{T}$, we only have to prove that $\sigma(u)$ is a semigroup, namely that it is closed under product. In fact, it is a classic result first proved by Numakura in \cite{Numakura} that a compact cancellative semigroup is in fact a group, and $\sigma(u)$ is definitely such a semigroup, for it is contained in $\mathbb{T}$. To this aim, let $\lambda, \mu\in\sigma(u)$. By virtue of a theorem of Brown and Pearcy, see \cite{Brown}, the spectrum of the Hilbert tensor product of two bounded operators on some Hilbert space is simply given by the product of their spectra, therefore $\lambda \mu\in\sigma(u\otimes u)$. As $\Delta_G(u)=u\otimes u$, we also have $\lambda\mu\in \sigma(\Delta_G(u))$. Because $\Delta_G$ is a *-homomorphism, the inclusion $\sigma(\Delta_G(u))\subset\sigma(u)$ holds, hence $\lambda\mu\in \sigma(u)$, which is what was to be proved.
\end{proof}
The theorem above applies notably to connected compact quantum groups, namely those groups whose dense $*$-Hopf subalgebra contains no finite dimensional $*$-Hopf subalgebras, see \cite{CDPR} for fuller information. Indeed, we have the following corollary.
\begin{corollary}
If $u$ is a non-trivial unitary one-dimensional representation of a connected compact quantum group $G=(Q_G,\Delta_G)$, then $\sigma(u)=\mathbb{T}$.
\end{corollary}
\begin{proof}
Let $Q_u\subset Q_G$ the $C^*$-subalgebra generated by $u$. From $\Delta_G(u)=u\otimes u$ the inclusion $\Delta_G(Q_u)\subset Q_u\otimes Q_u$ follows almost immediately. This says that $Q_u\subset Q_G$ is a $*$-Hopf-subalgebra. As such, it must be infinite-dimensional thanks to connectedness. Since $Q_u$ is isomorphic with $C(\sigma(u))$, the spectrum of $u$ is then forced to be an infinite set. Accordingly, $\sigma(u)$ is an infinite closed subgroup of $\mathbb{T}$. The thesis now easily follows, for $\mathbb{T}$ is the only infinite closed subgroup of $\mathbb{T}$.
\end{proof}
Interesting examples  of quantum groups are provided by discrete groups. Indeed, if $\Gamma$ is any such group, the reduced $C^*$-algebra $C_r^*(\Gamma)$ can be easily turned into a compact quantum group, of which $\Gamma$ should be thought of as the dual object. Moreover, $C_r^*(\Gamma)$ turns out to be connected precisely when $\Gamma$ is a torsion-free group, in a way that closely parallels a classical well-known result of Pontryagin, cf. \cite{CDPR} and the references therein.  Before moving to the next result,
we briefly recall what the reduced $C^*$-algebra of a discrete group $\Gamma$ is, as to fix the notations needed to follow the relative proof. This is simply the concrete $C^*$-subalgebra of $B(\ell_2(\Gamma))$ generated by the unitary operators $u_\gamma$, where $\gamma$ runs over $\Gamma$, whose action on the canonical orthonormal basis $\{\delta_\alpha:\alpha\in\Gamma\}$ of $\ell_2(\Gamma)$ is given by left translation, i.e. $u_\gamma \delta_\alpha\doteq \delta_{\gamma\alpha}$.
\begin{corollary}
Let $\Gamma$ be a torsion-free discrete group and let $C_r^*(\Gamma)$ the corresponding reduced C*-algebra. Then $\sigma(u_{\gamma})=\mathbb{T}$ for any $\gamma\in \Gamma$.
\end{corollary}
\begin{proof}
We shall think of $C_r^*(\Gamma)$ as a cocommutative compact quantum group whose coproduct is simply obtained by extending $\Delta_{\Gamma}(u_{\gamma})=u_{\gamma}\otimes u_{\gamma}$. By its very definition, each $u_{\gamma}$ is a unitary
one-dimensional representation. Since $\Gamma$ is a torsion free, the pair $(C_r^*(\Gamma), \Delta)$ is a connected quantum group, hence the foregoing corollary applies to $u_{\gamma}$.
\end{proof}
\begin{remark}
Besides the reduced $C^*$-algebra, one can also consider the universal $C^*$-algebra of a discrete group $\Gamma$. This is the enveloping $C^*$-algebra of the convolution algebra $\ell^1(\Gamma)$, which is only a unital involutive Banach algebra. It is usually
denoted by $C^*(\Gamma)$. Since by definition there exists a surjective $^*$-homomorphis $\pi: C^*(\Gamma)\rightarrow C^*_r(\Gamma)$, the above result applies of course to $u_\gamma$ understood as an element of $C^*(\Gamma)$ as well.
\end{remark}
For instance, if we apply the former corollary to $\mathbb{Z}$ we recover the very well-known fact that the bilateral shift $U$ on $\ell_2(\mathbb{Z})$ has full spectrum. However, to the best of our knowledge there seems to be no straightforward way to prove the result in general without resorting to some basic quantum group theory, as we did. In fact, this is arguably an interesting if possibly not unprecedented example of a result on quantum groups shedding light on spectral theory. We feel this is worth mentioning here because it is operator algebra techniques that are generally helpful in developing the theory of compact groups and not the other way around. Obviously, there do exist different 
instances of interplay of this sort, such as the result pointed out by Skandalis that a compact quantum group is coamenable if and only if $\textrm{dim}(u)$ is in the spectrum of $\chi_u$ for any representation $u$, where $\chi_u$ is the associated character. Moreover, the former result might also be a little step towards the solution in the positive of the long-standing Kadison-Kaplansky conjecture, which asks whether the reduced $C^*$-algebra of a torsion-free group is projectionless, i.e. devoid of non-trivial projections. Of course this would amount to proving that the spectrum of every unitary is a connected subset of the one-dimensional torus, which is unfortunately out of the reach of our result. In fact, it only says this is positively true of a remarkable class of unitaries, the $u_{\gamma}$'s we dealt with above. However, for those unitaries the result tells us a bit more. For not only is their spectrum shown to be connected but it is also thoroughly computed.  As for the conjecture, we cannot not say a word or two about the present state of the art. This must needs include the basic information that $C^*_r(\mathbb{F}_2)$ was first proved to be projectionless by Cohen in \cite{Cohen}, where $\mathbb{F}_2$ is the free group on two generators. Later proofs then followed, such as those of Pimsner-Voiculescu \cite{Pimsner}, Connes \cite{Connes}, and Effros \cite{Effros}. As well as the single but important example of $\mathbb{F}_2$, vast classes of groups are also known for which the corresponding reduced $C^*$-algebras are projectionless. Notably, groups of polynomial growth all satisfy the conjecture, as proved by Ji \cite{Ji} in 1992. Ten years later all word-hyperbolic groups, too, were eventually shown to satisfy the conjecture in a deep work by Puschnigg \cite{Puschnigg}.\\

Going back to the last result, it is worthwhile to observe that it continues to hold for higher-dimensional representations of cocommutative groups.
\begin{corollary}
If $u$ is any non-trivial representation of a connected cocommutative quantum group $G$, then $\sigma(u)=\mathbb{T}$.
\end{corollary}
\begin{proof}
By complete reducibilty, we can rewrite $u$ as a direct sum of irreducible subrepresentations, say $u=u_1\oplus u_2\oplus\dots\oplus u_n$. Since $G$ is cocommutative, each $u_i$ is one-dimensional. In addition, at least  one of them must be non-trivial. The thesis now follows from $\sigma(u)=\cup_{i=1}^n\sigma(u_i)$, which is immediately verified.
\end{proof}
As a matter of fact, the above result cannot be generalized any further. Indeed, for a higher-dimensional unitary representation $u\in U(M_n(Q_G))$ it is no longer true that $\sigma(u)\subset\mathbb{T}$ is a subgroup. For the property fails to be true already in the classical case, where $\sigma(u)$ is  nothing but the set $\bigcup_{g\in G}\sigma(u(g))$.\\ 
%However, if $u$ is a non-trivial unitary representation of a (classical) connected compact Lie group, then $\sigma(u)$ is still the whole one-dimensional torus $\mathbb{T}$ irrespective of dim($u$). This is obtained effortlessly by making a care use of a maximal torus $H\subset G$. We first need to observe, though, that any non-trivial representation of a torus
%$\mathbb{T}^k$ has full spectrum, as immediately checked.  We can then get to the result through the chain of inclusions $\mathbb{T}\supset\sigma(u)\supset\sigma(u\upharpoonright_H)=\mathbb{T}$, where the last equality is simply due to $u\upharpoonright_H$ being non-trivial as well.\\

Finally, we would like to conclude by giving yet another application of our main result. This is a simple proof that in the topolgical framework of Woronowicz compact quantum groups do not have non-trivial primitive elements, which should be a well-known fact, although not to be easily found in the literature.   
\begin{proposition}
Let $G=(Q_G,\Delta)$ be a compact quantum group. If $a\in Q_G$ is a self-adjoint element such that $\Delta(a)=a\otimes 1+1\otimes a$, then $a=0$.
\end{proposition}
\begin{proof}
If $a=a^*$ satisfies $\Delta(a)=a\otimes 1+1\otimes a$, then $ta$ satisfies the same relation for every $t\in\mathbb{R}$. Therefore, $u_t\doteq e^{ita}$ is a one-parameter group of unitary one-dimensional representations of $G$. In particular, $\sigma(u_t)\subset \mathbb{T}$
is a subgroup for each $t$. Since $\sigma(ta)\subset[-t\|a\|, t\|a\|]$, by the spectral mapping theorem we also have $\sigma(u_t)\subset\{e^{i\theta}: \theta\in[-t\|a\|, t\|a\|]\}$. Hence, if $t$ is small enough, $\sigma(u_t)$ cannot be a sugroup, unless it is zero, in which case $u=1$, that is to say $a=0$.
\end{proof}
\begin{remark}
The classical counterpart of the above proposition needs no proof. Indeed, it reduces to the easily verified fact that any continuous homomorphism from a compact group $G$ into $(\mathbb{R},+)$ is trivial.
\end{remark}
In purely algebraic contexts, nevertheless, Hopf algebras will have a wealth of primitive elements. The simplest examples are provided by cocommutative algebras, whose structure is completely described by a well-known result of Cartier, Gabriel and Kostant.
This says that any cocommutative Hopf algebra $H$ over an algebraically closed field $\mathbb{K}$ of characteristic zero 
is in fact a suitable product of  $\mathbb{K}G$, the group algebra of a discrete group $G$, and $U(\mathfrak{g})$, the universal enveloping algebra of a Lie algebra $\mathfrak{g}$ acted upon by $G$, which is recovered, no surprise, as the set of the group-like elements of $H$.\\ 

\emph{Acknowledgments} I owe a debt of gratitude to Alessandro D'Andrea and Claudia Pinzari, without the advise and suggestions of whose this work would not have come into existence.

\end{document}